\documentclass{article}
\usepackage{theorem}
\usepackage{eurosym}
\usepackage{amssymb}
\usepackage{amsmath}
\usepackage{amsfonts}
\usepackage[margin=4cm]{geometry}
\usepackage[american]{babel}
\usepackage{verbatim}
\usepackage{graphicx}
\usepackage{color}%
\providecommand{\U}[1]{\protect\rule{.1in}{.1in}}

\newtheorem{theorem}{Theorem}

{\theorembodyfont{\rmfamily}

}

\newtheorem{lemma}[theorem]{Lemma}

\newtheorem{proposition}[theorem]{Proposition}
{\theorembodyfont{\rmfamily}
\newtheorem{remark}[theorem]{Remark}
}

\newenvironment{proof}[1][Proof]{\noindent\textbf{#1} }{\ \rule{0.5em}{0.5em}}

\newcommand*\re{\mathbb{R}}

\newcommand*\Omegabar{\overline{\Omega}}
\newcommand*\delomega{\partial\Omega}
\newcommand*\intdelomega{\int_{\partial\Omega}}

\newcommand*\Hol{\operatorname{Hol}}
\newcommand*\Real{\operatorname{Re}}
\newcommand*\Imag{\operatorname{Im}}

\date{}

\begin{document}

\title{An Isoperimetric Problem With Density and the Hardy Sobolev Inequality in $\re^2$}

\maketitle

\centerline{\scshape Gyula Csat\'{o} }
\medskip
{\footnotesize
 \centerline{Tata Institute of Fundamental Research, Centre For Applicable Mathematics, 560065 Bangalore, India}
   \centerline{gy.csato.ch@gmail.com}

}

\smallskip

\begin{abstract}
We prove, using elementary methods of complex analysis, the following generalization of the isoperimetric inequality: if $p\in\re$, $\Omega\subset\re^2$ then the inequality
$$
  \left(\frac{|\Omega|}{\pi}\right)^{\frac{p+1}{2}}\leq\frac{1}{2\pi}\int_{\delomega}|x|^pd\sigma(x)
$$
holds true under appropriate assumptions on $\Omega$ and $p.$ This solves an open problem arising in the context of isoperimetric problems with density and poses some new ones (for instance generalizations to $\re^n$).
We prove the equivalence with a Hardy-Sobolev inequality, giving the best constant, and generalize thereby the equivalence between the classical isoperimetric inequality and the Sobolev inequality. 
Furthermore, the inequality paves the way for solving another problem: the generalization of the harmonic transplantation method of Flucher to
the singular Moser-Trudinger embedding.
\end{abstract}

\let\thefootnote\relax
\let\thefootnote\relax

\section{Introduction}

Let $\Omega\subset\re^2$ be a bounded open simply connected set with a piecewise $C^1$ boundary $\partial\Omega.$ Suppose that $0\in\Omega,$ and let $p\geq -1.$ The following inequality is our main result:
\begin{equation}
 \label{eq:main isop. ineq. intro.}
  \left(\frac{|\Omega|}{\pi}\right)^{\frac{p+1}{2}}\leq\frac{1}{2\pi}\int_{\delomega}|x|^pd\sigma(x),
\end{equation}
where $|\Omega|$ denotes the $2$ dimensional volume of $\Omega$ and $d\sigma(x)$ is the $1$ dimensional Hausdorff measure. All other results will be deduced from this inequality. Setting $p=0$ we obtain the classical isoperimetric inequality. The proof is very simple, self-contained and uses only basic properties of holomorphic functions and the Riemann mapping theorem. This is a new approach for tackling an isoperimetric problem with density.

Isoperimetric problems with densities, or also called weights, are being investigated with increasing interest in the recent years. The general question is the following: Given two positive functions $f,g$ and defining the weighted volume $V_f$ and weighted perimeter $P_g$ as
$$
  V_f(\Omega)=\int_{\Omega}f(x)dx,\qquad P_g(\Omega)=\int_{\delomega}g(x)d\sigma(x),
$$
one studies the existence of minimizers of
\begin{equation}\label{intro:minimization classical with density}
  I(C)=\inf\left\{P_g(\Omega):\,\Omega\subset\re^2\text{ and }V_f(\Omega)=C\right\}.
\end{equation}
There are many recent results with different kinds of constraints on $f$ and $g,$ see for instance  \cite{Rosales Canete Bayle Morgan}, \cite{Figalli-Maggi.LogConvex}, \cite{Figalli-Maggi-Pratelli},
\cite{Fusco-Maggi-Pratelli}, \cite{Morgan-Pratelli} and the references therein. Motivated by two important applications in the theory of partial differential equations it turns out to be natural to investigate the following special case: let $p>-1$ and $q>-2$ and consider
\begin{equation}\label{intro:minimization scaled new with density}
  I(C)=  \inf\left\{\int_{\delomega}|x|^pd\sigma
  :\,\Omega\subset\re^2\text{ and }\left(\int_{\Omega}|x|^qdx\right)^{\frac{p+1}{q+2}}=C\right\}.
\end{equation}
In contrast to \eqref{intro:minimization classical with density} one of the main differences is that this problem is scaling invariant in the sense that if $\Omega$ is a minimizer for $I(C),$ then  $\lambda\Omega$ is a minimizer for $I(\lambda^{p+1}C).$ Hence, for the study of optimal shapes the value of $C$ becomes irrelevant and can be fixed. In the case $p=q,$ the combination of the works of several authors \cite{Canete-Miranda-Wittone}, \cite{Carroll-Jacob-Quinn-Walters}, \cite{Morgan-Regularity} and finally \cite{Dahlberg-Dubbs-Newkirk-Tran} have led to the solution, which states that balls whose boundary contains the origin are the minimizers. D\'iaz-Harman-Howe-Thompson \cite{Diaz-Harman-Howe-Thompson} are able to extend this to some more general results (see \cite{Diaz-Harman-Howe-Thompson} Proposition 4.21), but the range of values of $p$ and $q$ is such that  the classical isoperimetric inequality does not appear as a special case. This is precisely one of 
the remaining open problems, see \cite{Diaz-Harman-Howe-Thompson} Open Questions 1.5 (4), i.e. $q=0$ and $p\in(0,1).$ Our Theorem \ref{theorem:main theorem weighted isop. inequlity} (ii) settles precisely this problem. They also conjecture many more results for more general values of $p$ and $q,$ see 
\cite{Diaz-Harman-Howe-Thompson} Conjecture 4.22. 

Another closely related work is by Betta-Brock-Mercaldo-Posteraro \cite{Betta and alt.}, whose result implies the case $q=0$ and $p\geq 1.$  We will state and use this result (see Theorem \ref{theorem:Betta et alt}) to simplify our proof for the case $p\geq 1$ and to give a complete picture. In their proof one adds the  classical isoperimetric inequality to rather generous estimates arising from a weight which satisfies a monotonicity and convexity condition. We point out that the relevant applications for PDE in this paper are for the range $p\in [-1,1],$ where convexity fails and which will be precisely those including the classical isoperimetric inequality as a special case.

It is well known since the paper of Talenti \cite{Talenti} that the Sobolev embedding, with the best constant,
\begin{equation}\label{eq:intro:Sobolev classical}
  \|u\|_{L^2}\leq\frac{1}{2\sqrt{\pi}}\int_{\re^2}|\nabla u(x)|dx\quad\text{ for all }u\in C_c^{\infty}\left(\re^2\right)
\end{equation}
is equivalent to the isoperimetric inequality. This equivalence is essentially due to Federer-Fleming \cite{Federer-Fleming} and Fleming-Rishel \cite{Fleming-Rishel}, by means of the coarea formula. We generalize this equivalence to the Hardy-Sobolev inequality. To the knowledge of the author, such a proof of the Hardy-Sobolev, respectively Caffarelli-Kohn-Nirenberg inequalities, is new. Moreover, it gives the best constant. More precisely it turns out that for values 
$-2<p-1\leq q\leq 2p$
the problem \eqref{intro:minimization scaled new with density} is equivalent with determining the best constant in the following case of the
Caffarelli-Kohn-Nirenberg \cite{Caffarelli-Kohn-Nirenberg} inequality:
\begin{equation}
 \label{intro main. Caff-Kohn-Ni ineq.}
  \|u\,|x|^{\gamma}\|_{L^r}\leq C\|Du\,|x|^{\alpha}\|_{L^1}\quad\text{ for all }u\in C_c^{\infty}(\re^2),
\end{equation}
where $\alpha,\gamma$ and $r$ satisfy certain compatibility conditions, see \eqref{eq:conditions in Caff-Kohn-Ni ineq.}.
If $q=0,$ the range $p\in (-1,1]$ is the best possible to establish such an equivalence with this method. For $p\in [0,1]$ we obtain precisely \eqref{intro main. Caff-Kohn-Ni ineq.}, whereas for $p\in (-1,0)$ the set of admissible functions $C_c^{\infty}(\re^2)$ has to be restricted appropriately. For instance, one of our results states that if $p\in [0,1]$ then 
$$
  \|u\|_{L^r(\re^2)}\leq \frac{\pi^{1/r}}{2\pi}\int_{\re^2}|\nabla u(x)|\,|x|^p dx,\quad\forall\,u\in C_c^{\infty}\left(\re^2\right),\quad r=\frac{2}{p+1}.
$$
Setting $p=0$ gives \eqref{eq:intro:Sobolev classical}.
Best constants and extremal functions have been investigated for all kinds of cases for the Sobolev and Caffarelli-Kohn-Nirenberg inequality and the literature is very large, see for instance \cite{Catrina-Wang}, \cite{Chou and Chu}, \cite{CorderoVillani et alt.}, \cite{Horiuchi} and \cite{Lieb}, and the references therein. However, the best constant has not yet been established in the present case. 

Our original motivation for the weighted isoperimetric inequality, in particular the range $p\in [-1,0],$ came from the Moser-Trudinger inequality.
The existence of an extremal function for the Moser-Trudinger embedding when $\Omega=B_1$ (unit ball) has been shown by Carleson-Chang \cite{Carleson-Chang}; see also the recent result of Malchiodi-Martinazzi \cite{Malchiodi-Martinazzi} in dimension $2.$ Flucher \cite{Flucher} proved the so-called \textit{functional isoperimetric inequality} to generalize the result to arbitrary domains $\Omega\subset\re^2.$ As the name suggests, Flucher's proof uses the isoperimetric inequality in a crucial way. \eqref{eq:main isop. ineq. intro.} will allow us to generalize this method to more general functionals involving weights, such as (cf. Adimurthi-Sandeep \cite{Adi-Sandeep})
$$
  u\in W_0^{1,2}(\Omega)\to \int_{\Omega}\frac{e^{\alpha u^2}}{|x|^{\beta}}dx.
$$
where $\Omega\subset\re^2$ is a domain containing the origin and $\alpha/(4\pi)+\beta/2\leq 1.$ We have established the existence of extremal functions in a separate paper \cite{Csato-Roy} and restrict ourselves here to generalizing only that crucial inequality where the isoperimetric inequality is used: If $G_{\Omega,x}$ is the Green's function with singularity at $x\in\Omega,$ then Flucher \cite{Flucher} has shown that $|\Omega|$ can be estimated by
$$
  |\Omega|\leq\frac{1}{4\pi}\intdelomega\frac{1}{|\nabla G_{\Omega,x}(y)|}d\sigma(y).
$$
In Theorem \ref{theorem:upper bound for R by Greens function with weight} we will prove a generalization of this estimate. 

It is clear that many of the ideas presented in this paper give rise to possible generalizations: higher dimensions and other values of $p$ and $q.$

\section{A Weighted Isoperimetric Inequality}

The following is the main theorem of this section.

\begin{theorem}
\label{theorem:main theorem weighted isop. inequlity}
Let $\Omega\subset\re^2$ be a bounded open set with piecewise $C^1$ boundary $\delomega.$  Regarding the inequality
\begin{equation}
 \label{theorem:eq:main isop. ineq.}
  \left(\frac{|\Omega|}{\pi}\right)^{\frac{p+1}{2}}\leq\frac{1}{2\pi}\int_{\delomega}|x|^pd\sigma,
\end{equation}
the following statements hold true:
\smallskip

%(i) If $p\in\re,$ then \eqref{theorem:eq:main isop. ineq.} holds if $\Omega$ is %simply connected and $0\in\Omega.$ 
%\smallskip

(i) If $p\geq -1,$ then \eqref{theorem:eq:main isop. ineq.} holds if $\Omega$ is connected and $0\in\Omega.$ 
\smallskip 

(ii) If $p\geq 0,$ then \eqref{theorem:eq:main isop. ineq.} holds for all $\Omega.$
\smallskip

(iii) If $p\geq -1,$ $\Omega$ satisfies the conditions of (i), $\partial\Omega$ is $C^2$ and there is equality in \eqref{theorem:eq:main isop. ineq.}, then $\Omega$ is a ball. If $p\neq 0,$ then this ball must be centered at the origin.
\smallskip

(iv) If $p>0,$ $0\notin\delomega,$ $\delomega$ is $C^2,$ and there is equality in \eqref{theorem:eq:main isop. ineq.}, then $\Omega$ is a ball centered at the origin. If in addition $p\geq 1,$ then the same conclusion holds without the assumption $0\notin\delomega.$
\end{theorem}

\begin{remark}
\label{remark:after main theorem for weighted isop ineq.}
(a) We will sometimes use the following equivalent formulation: let $B_R=B_R(0)$ be the ball of radius $R$ and center at the origin such that $|\Omega|=|B_R|=\pi R^2.$ Then we have 
\begin{equation}
 \label{eq:remark to main theorem:equiv formulation with BR}
  \int_{\delomega}|x|^pd\sigma\geq\int_{\partial B_R}|x|^pd\sigma=2\pi R^{p+1}.
\end{equation}
Or in other words: Among all admissible sets with given fixed volume the set $B_R(0)$ is a minimizer of the weighted perimeter.
\smallskip

(b) If $-1\leq p<0,$ neither the assumption that $\Omega$ is connected, nor that $0\in\Omega$ can be dropped. Thus the hypothesis in (i) is optimal. For $\Omega$ not connected, take the following example: fix $r>0$ and $\Omega$ is the union of two disconnected balls $\Omega=B_r(0)\cup B_r(y).$ Obviously $0\in\Omega.$ For $|y|$ large enough we get a contrdiction.
For the case $0\notin \Omega$ one can argue in a similar way.

\smallskip
(c) The condition $0\notin\delomega$ in Part (iv) if $p\in(0,1)$ is for technical reasons, due probably to the method of proof. The hypothesis $0\notin\delomega$ can clearly be removed if $p\geq 1,$ in view of Theorem 4.3 in \cite{Betta and alt.}.
\end{remark}

The case $p\geq 1$ is already known since it follows from Betta-Brock-Mercaldo-Posteraro \cite{Betta and alt.} Theorem 2.1. We state here their result in dimension 2, however it is also valid in higher dimensions.

\begin{theorem}[Betta-Brock-Mercaldo-Posteraro]\label{theorem:Betta et alt} Let $\Omega\subset\re^2$ be a bounded open set with Lipschitz boundary. Let $R>0,$ be such that $|\Omega|=\pi R^2.$ Let a be a function $a:[0,\infty)\to [0,\infty)$ with the following two properties
\begin{equation}
  \label{assumptions on a in Betta and alt.}
(i)\quad
   a \quad\text{ is nondecreasing}\qquad (ii)\quad
   z\mapsto \left(a\left(\sqrt{z}\right)-a(0)\right)\, \sqrt{z}\quad\text{ is convex}.
\end{equation}
Then
$$
  \intdelomega a(|x|)d\sigma(x)\geq \int_{\partial B_R(0)} a(|x|)d\sigma(x)
  =2\pi R\,a(R).
$$
\end{theorem}

We illustrate the idea of the proof of this theorem on the example $a(z)=C+z^p,$  $C\in\re,$ $p\geq 1$ and for a domain starshaped with respect to the origin. Under these assumptions $\delomega$ can be parametrized in the form 
$\theta\in [0,2\pi]\mapsto r(\theta)(\cos\theta,\sin\theta)$ for some $2\pi$ periodic functin $r>0.$
Using polar coordinates we get
$$
  \pi R^2=|\Omega|=\int_0^{2\pi}\int_0^{r(\theta)}s dsd\theta=\frac{1}{2}\int_0^{2\pi} r^2(\theta)d\theta.
$$
Let us first assume that $C=0.$ Using that $p\geq 1$ implies that the map $g:(0,\infty)\to \re$ given by
$ g(s)=s^{\frac{p+1}{2}}$ is convex. We therefore obtain from Jensen inequality that
\begin{equation}
  \label{eq:sketch proof of betta et alt.}
  \begin{split}
  2\pi R^{p+1}=&2\pi\left(\frac{1}{2\pi}\int_0^{2\pi} r^2(\theta)d\theta\right)^{\frac{p+1}{2}}
  \leq
  \int_0^{2\pi}r^{p+1}(\theta)d\theta 
  \smallskip \\
  \leq & \int_0^{2\pi}r^p(\theta)\sqrt{r'^2(\theta)+r^2(\theta)}d\theta
  =\int_{\delomega}|x|^pd\sigma(x).
 \end{split}
\end{equation}
If $C\neq 0,$ then one adds the classical isoperimetric inequality to \eqref{eq:sketch proof of betta et alt.} and the same proof works.
The disadvantage of this result is that all our relevant applications are for $p\in [-1,1],$ in particular the classical isoperimetric inequality.

\smallskip

We now prove Part (i) of Theorem \ref{theorem:main theorem weighted isop. inequlity}. This is the central part of the proof and the most difficult case. If we set $p=0$ everywhere in the following proof we recover the proof sketched by Mateljevic-Pavlovic \cite{Mateljevic-Pavlovic} for the classical isoperimetric inequality. Part (ii) will be deduced from (i) by fairly obvious geometric arguments and elementary inequalities, ideas which already appear in the same or similar form in \cite{Carroll-Jacob-Quinn-Walters} (Proposition 4.5) and \cite{Morgan-Pratelli} (Proposition 6.10). In what follows we can always assume, by approximation, that $\delomega$ is $C^2.$ If $U\subset \mathbb{C},$ then $\operatorname{Hol}(U)$ will denote the set of holomorphic functions in $U.$
\smallskip

\begin{proof}[Proof (Theorem \ref{theorem:main theorem weighted isop. inequlity} Part (i)).] 
\textit{Step 1.} Let us first assume that $\Omega$ is simply connected and let $\overline{\Omega}^c=\mathbb{C}\backslash\overline{\Omega}$ denote the complement of $\overline{\Omega}.$ Consider the set $D\subset \mathbb{C}$ defined by
$$
  D=\left\{\frac{1}{z}:\; z\in\overline{\Omega}^c\right\}\cup\{0\}.
$$
Note that the map $\eta(z)=1/z$ is one-to-one and $\partial\Omega$ is a simple closed curve. Therefore $\eta$ maps $\delomega$ onto a simple closed  curve $\Gamma$. By the Jordan curve theorem the interior of $\Gamma$ is a bounded simply connected set. It must coincide with $D$ since $\eta$ maps $\infty$ to $0.$ Thus $D$ is a bounded open simply connected set with $C^2$ boundary such that $0\in D.$ By the Riemann mapping theorem there exist a conformal map $h:B_1\to D$ (which extends to a $C^1$ diffeomorphism on $\overline{B}_1,$ cf. for instance Theorem 5.2.4 page 121 in Krantz \cite{Krantz Complex Analysis}) such that
$$
  h(0)=0\quad\text{ and }\quad h\in \operatorname{Hol}(B_1)\cap C^1(\overline{B}_1;\overline{D}).
$$
Define now $g\in \operatorname{Hol}(B_1\backslash\{0\})\cap C^1(\overline{B}_1\backslash\{0\})$ by
$$
  g=\frac{1}{h}:\overline{B}_1\backslash\{0\}\to \Omega^c.
$$
Since $h(0)=0,$ there exists $G\in\Hol(B_1)$ such that 
$
  h(z)=zG(z)\quad\text{ for all }z\in B_1\,.
$
Because $h'(0)\neq 0,$ we must have that $G(0)\neq 0.$ Moreover, using again that $h(0)=0$ and that $h$ is one-to-one, we get that $G(z)\neq 0$ for all $z\in B_1.$  Moreover $h(z)\neq 0$ for all $z\in\partial B_1$ and therefore 
$$
  G(z)\neq 0\quad\text{ for all }z\in\overline{B}_1\quad\text{ and }\quad
  G\in \Hol(B_1)\cap C^1(\overline{B}_1).
$$
Finally we define $Q\in \Hol(B_1)\cap C^1(\overline{B}_1)$ and $\lambda\in\mathbb{C}$ by
$$
  Q(z)=\frac{1}{G(z)}\quad\text{ and }\quad \lambda=Q(0).
$$
Note that $\lambda\neq 0.$
In view of the above definitions, there exists a holomorphic function $P=(Q-\lambda)/z\in \Hol(B_1)$ such that
\begin{equation}
  \label{eq:proof:def of P}
   g(z)=\frac{1}{h(z)}=\frac{Q(z)}{z}=\frac{\lambda}{z}+P(z)\quad\text{ for all }z\in B_1\,.
\end{equation}
\smallskip

\textit{Step 2.} We will prove in this step that
\begin{equation}
 \label{eq:proof: omega volume smaller the pi lambda2}
  |\Omega|\leq \pi|\lambda|^2.
\end{equation}

\textit{Step 2.1.} The map $h:\partial B_1$ to $\partial D$ is one-to-one and onto. The same holds for the map $\eta(z)=1/z,$ which maps $\partial D$ to $\delomega.$ Therefore the curve $\gamma(t)=g(e^{-it}),$ $t\in[0,2\pi],$ 
is a parametrization of $\delomega.$ Since $\eta$ inverses the orientation of the curve, we have taken $e^{-it}$ (instead of $e^{it}$) and therefore $|\Omega|$ computes as
$$
  |\Omega|=\frac{1}{2}\int_{\Omega}\operatorname{div}(x_1,x_2)dx
  =\frac{1}{2}\int_0^{2\pi} \left(\gamma_1(t)\gamma_2'(t)-\gamma_2(t)\gamma_1'(t)\right) dt,
$$
where $\gamma_1(t)=\Real(\gamma(t))$ and $\gamma_2(t)=\Imag(\gamma(t)).$ Since $\gamma$ is not contained in the open set $B_1$ we do an approximation and define $\gamma_r(t)=g(re^{-it}),$ where $0<r<1,$ and define
$$
  A_r=\frac{1}{2}\int_0^{2\pi} \left((\gamma_r)_1(t)(\gamma_r)_2'(t)-(\gamma_r)_2(t)(\gamma_r)_1'(t)\right) dt.
$$
Since $g\in C^1(\overline{B}_1\backslash\{0\}),$ we have that
\begin{equation}
\label{eq:proof:Ar goes to omega}
  \lim_{r\to 1}A_r=|\Omega|.
\end{equation}

\textit{Step 2.2.} In this Step we will estimate $A_r$. We have by definition of $\gamma_r(t)$
$$
  \gamma_r(t)=g\left(r e^{-it}\right)\quad\text{ and }\quad
  \gamma_r'(t)=-ig'\left(r e^{-it}\right)re^{-it}.
$$
Therefore if we set 
\begin{align*}
  z=&g\left(r e^{-it}\right)=\frac{\lambda}{r}e^{it}+P\left(r e^{-it}\right)
  \smallskip \\
  w=&-g'\left(r e^{-it}\right)r^{-it} =\frac{\lambda}{r}e^{it}-r e^{-it}P'\left(r e^{-it}\right),
\end{align*}
and use the formula:
$
  \Real(z)\Imag(iw)-\Imag(z)\Real(iw)=\Real(z\overline{w}),
$
we get that 
\begin{equation}
 \label{eq:proof:Ar in terms of z and w}
  A_r=\frac{1}{2}\int_{0}^{2\pi}\Real(z\overline{w})dt
  =\frac{1}{2}\Real\int_{0}^{2\pi}z\overline{w}\,dt.
\end{equation}
Since $P$ is holomorphic in $B_1,$ we can write $P$ as
$
  P(z)=\sum_{n=0}^{\infty}a_nz^n,
$
where the series converges absolutely and uniformly on every $\overline{B}_r,$ and thus in particular on $\partial B_r$. If we define $a_{-1}=\lambda,$ then $z$ and $\overline{w}$ can be written as
\begin{align*}
  z=&\frac{\lambda}{r}e^{it}+\sum_{n=0}^{\infty}a_nr^ne^{-int}= \sum_{n=-1}^{\infty}a_nr^ne^{-int}
  \smallskip \\
  \overline{w}=&\frac{\overline{\lambda}}{r}e^{-it}-\sum_{m=1}^{\infty} m\overline{a}_mr^m e^{imt} 
  =-\sum_{m=-1}^{\infty} m\overline{a}_mr^m e^{imt}.
\end{align*}
We therefore obtain that
\begin{align*}
  \int_0^{2\pi}(z\overline{w})dt=& -\int_0^{2\pi}\sum_{n,m=-1}^{\infty}m a_n\overline{a}_m r^{n+m}e^{i(m-n)t} dt 
  \smallskip \\
  =&-\sum_{n,m=-1}^{\infty}m a_n\overline{a}_m r^{n+m} \int_0^{2\pi}e^{i(m-n)t} dt
  = -2\pi\sum_{n=-1}^{\infty}n|a_n|^2 r^{2n}.
\end{align*}
We recall that $a_{-1}=\lambda$ and set the previous equation into \eqref{eq:proof:Ar in terms of z and w} to get
\begin{equation}
 \label{eq:proof of main theorem. Final Ar}
  A_r=-\pi \sum_{n=-1}^{\infty}n|a_n|^2 r^{2n}
  =\pi\left(\frac{|\lambda|^2}{r^2}-\sum_{n=1}^{\infty}n|a_n|^2 r^{2n}\right)
  \leq\frac{\pi |\lambda|^2}{r^2}.
\end{equation}
Hence this inequality with \eqref{eq:proof:Ar goes to omega} and letting $r\to 1$ proves the claim of Step 2.
\smallskip

\textit{Step 3.} We will prove in this step that
$$
  |\lambda|^{p+1}\leq \frac{1}{2\pi}\int_{\delomega}|x|^p d\sigma(x).
$$
As in Step 2, we get that $t\in[0,2\pi]\to\alpha(t)=g(e^{it})$ is a parametrization of $\delomega.$ Using that $|\alpha'(t)|=|g'(e^{it})ie^{it}|=|g'(e^{it})|,$ gives
$$
  \int_{\delomega}|x|^pd\sigma(x)=\int_0^{2\pi}\left|g'\left(e^{it}\right)\right| \,\left|g\left(e^{it}\right)\right|^p dt.
$$
As in Step 2, we set $\alpha_r(t)=g(re^{it})$ for $0<r<1.$ Observe that, since $g\in C^1(\overline{B}_1\backslash\{0\}),$ we get that
\begin{equation}
 \label{eq:proof:implicit def on Sr}
 \int_{\delomega}|x|^p d\sigma(x)=\lim_{r\to 1} S_r,
\end{equation}
where 
$$
  S_r=\int_0^{2\pi}\left|\alpha_r'(t)\right| \,\left|\alpha_r(t)\right|^pdt
  =
  r\int_0^{2\pi}\left|g'\left(re^{it}\right)\right| \,\left|g\left(re^{it}\right)\right|^pdt.
$$
Let us define $u,$ using \eqref{eq:proof:def of P}, by $u(z)=z^2g'(z)=zQ'(z)-Q(z).$ Note that 
$u\in \Hol(B_1)$ and $u(0)=\lambda.$
Since $B_1$ is simply connected and $Q$ does not vanish on $B_1$, there exists a holomorphic logarithm $\varphi$ of $Q,$ that is
$$
  \varphi\in \Hol(B_1)\quad\text{ and }\quad Q(z)=\exp(\varphi(z))\quad\text{ for all }z\in B_1\,.
$$
In particular the function $\tau$ defined by
$
  \tau(z)=u(z)\exp(p\varphi(z))
$
is a holomorphic function in $B_1$. We can therefore apply the Cauchy mean value integral formula
\begin{equation}\label{eq:tau at zero and cauchy formula}
  \tau(0)=\frac{1}{2\pi}\int_0^{2\pi}u\big(r e^{it}\big) \exp\big(p\varphi\big(r e^{it}\big)\big) dt.
\end{equation}
Note that for any $z\in\mathbb{C}$ and $p\in\re$ the identiy $|\exp(pz)|=|\exp(z)|^p$ holds true. The previous equality leads to the estimate
\begin{equation}
 \label{eq:proof:estimate of tau 0} 
  |\tau(0)|\leq \frac{1}{2\pi}\int_0^{2\pi}\big|u\big(re^{it}\big)\big|\,\big|Q\big(r e^{it}\big)\big|^p dt.
\end{equation}
From the definitions of $Q,$ $u$ and $\tau$ we get the following three identities
\begin{align*}
  |\tau(0)|=|u(0)|\,|Q(0)|^p=|\lambda|^{p+1}\quad&\text{ and }\quad
  \big|u\big(r e^{it}\big)\big|= r^2\big|g'\big(r e^{it}\big)\big|
  \smallskip \\
  \big|Q\big(re^{it}\big)\big|=&\big|re^{it}\big|\,\big|g\big(re^{it}\big)\big|
  =r\big|g\big(re^{it}\big)\big|
\end{align*}
Plugging these identities into \eqref{eq:proof:estimate of tau 0} yields
$$
  |\lambda|^{p+1}=|\tau(0)|\leq \frac{r^{p+2}}{2\pi}
  \int_0^{2\pi}\left|g'\left(re^{it}\right)\right| \,\left|g\left(re^{it}\right)\right|^pdt=\frac{r^{p+1}}{2\pi}S_r\,.
$$
Finally, letting $r\to 1$ and using \eqref{eq:proof:implicit def on Sr} proves the claim of Step 3.

\smallskip 

\textit{Step 4.} Sinc $p\geq-1,$ Steps 2 and 3 imply that
\begin{equation}\label{eq:main thm proof step 4}
  \left(\frac{|\Omega|}{\pi}\right)^{\frac{p+1}{2}}
  \leq |\lambda|^{p+1} \leq
  \frac{1}{2\pi}\int_{\delomega}|x|^p d\sigma(x),
\end{equation}
which proves the Part (i) the theorem in the case that $\Omega$ is simply connected. If $\Omega$ is not simply connected, 
then there exists an integer $m$ and simply connected open bounded sets $\Omega_i$, $i=0,\ldots,m,$ such that for $i=1,\ldots,m,$ the sets $\overline{\Omega}_i$ are disjoint and $\overline{\Omega}_i\subset\Omega_0,$  $0\in\Omega_0$ and
$$
  \Omega=\Omega_0\backslash \left(\bigcup_{i=1}^m \overline{\Omega}_i\right).
$$
Let $|\Omega_0|=\pi R_0^2$ and $|\Omega|=\pi R^2.$ Using Part (i) and that $p+1\geq 0$ we obtain that 
$$
  \int_{\delomega}|x|^pd\sigma(x)\geq\int_{\delomega_0}|x|^pd\sigma(x)\geq 2\pi R_0^{p+1}\geq 2\pi R^{p+1}.
$$
This proves Part (i) of the theorem.
\end{proof}
\smallskip

We now turn to the proof of Part (ii). For the case $0\notin\Omega$ the idea is very simple and consists roughly speaking of the following: If $p\geq0,$ then $|x|^p$ decreases for all $x\in\delomega,$ if $\Omega$ is shifted closer to the origin in an appropriate way. We introduce the following notation: let $a,b\in\re^2$ and
\begin{align*}
 [a,b]=&\left\{x\in\re^2;\, x=\lambda a+(1-\lambda) b,\;0\leq\lambda\leq 1\right\}
 \smallskip \\
 \overline{ab}=&\left\{x\in\re^2;\, x=\lambda a+(1-\lambda) b,\; \lambda\in\re\right\}.
\end{align*}
We will use the following simple lemma.

\begin{lemma}
\label{lemma:p bigger 0 straight line minimizes}
Let $a,b\in \re^2$ be such that $0\in\overline{ab}$ and assume that $p\geq 0.$ Then 
$$
  \min\left\{\int_{\gamma}|x|^pd\sigma;\,\gamma\in C^1\left([0,1],\re^2\right),\;\gamma(0)=a,\,\gamma(1)=b\right\}=\int_{[a,b]}|x|^pd\sigma.
$$
\end{lemma}

\begin{proof}
Projecting any curve $\gamma$ onto $\overline{ab}$ gives a new curve which decreases the integral. The proof is elementary and we omit the details.
\end{proof}

\smallskip

\begin{proof}[Proof (Theorem \ref{theorem:main theorem weighted isop. inequlity} Part (ii)).] \textit{Step 1.}
Note that if $p\geq 0,$ the map $x\mapsto|x|^p$ is continuous. We thus obtain from Part (i) and by a continuity argument that \eqref{theorem:eq:main isop. ineq.} holds
for all $\Omega$ such that $0\in\Omegabar.$ Moreover, as in the proof of Part (i) Step 4, we can assume that $\Omega$ is simply connected.
\smallskip

\textit{Step 2.} We show in this step that we can drop the assumption $0\in\Omega.$ So suppose that $0\notin\Omegabar,$ but we still assume that $\Omega$ is connected. Then we can also assume, as in the proof of Part (i) Step 4, that $\Omega$ is simply connected. Define $E=\operatorname{conv}(\Omegabar)$ as the convex hull of $\Omegabar.$ We now distinguish two cases.
\smallskip

\textit{Case 1: $0\notin E.$} By the definition of $E$ we can assert the existence of $x_0\in\partial E,$ $y_1,y_2\in\delomega$ and $0\leq\lambda\leq 1$ with the following properties
\begin{itemize}
 \item[(i)] $|x_0|=\min_{y\in E}|y|$ and $x_0=\lambda y_1+(1-\lambda)y_2$.
 \item[(ii)] $\overline{y_1y_2}$ is a separtating hyperplane for $\{0\}$ and $E,$ i.e.
$
   \langle x_0;x-x_0\rangle\geq 0$ for all $x\in E.
$
\end{itemize}
We have assumed $y_1\neq y_2,$ otherwise we can take $x_0\in\delomega$ and the following proof simplifies considerably, as can be easily verified.
It follows from property (ii), that $|x-x_0|\leq |x|$ for all $x\in E$ and therefore
\begin{equation}
 \label{eq:proof:shifting arg. x norm power p decreases}
  |x-x_0|^p\leq |x|^p\quad\text{ for all }x\in E.
\end{equation}
Since $\Omega$ is simply connected, $\delomega$ is the image of an oriented simple closed curve $\gamma.$ The points $y_1,y_2$ split $\gamma$ into two curves $\gamma_1$ and $\gamma_2$ which do not intersect (except at $y_1,y_2$), such that $\gamma=\gamma_1+\gamma_2$ and have the properties (reparametrizing, if necessary)
\begin{align*}
 \gamma_1(0)=y_1,\;\gamma_1(1)=y_2,&\qquad \gamma_2(0)=y_2,\;\gamma_2(1)=y_1\,.
\end{align*}
By abuse of notation, let us consider $[y_1,y_2]$ as an oriented curve. Then
the two closed simple curves $\gamma_1-[y_1,y_2]$ and $\gamma_2+[y_1,y_2]$ both bound a simply connected bounded set, let's say $\Omega_1$ and $\Omega_2$. Since $\overline{y_1y_2}$ is a separating hyperplane and $\gamma_1$ and $\gamma_2$ do not intersect, it follows that $\Omega$ must be contained in one of the $\Omega_1$ or $\Omega_2$ (It is enough to note that $[y_1,y_2]$ and $\gamma_1$ both join $y_1$ and $y_2$. So if there is a point $\gamma_2(t),$ $t\in(0,1),$ which lies in $\Omega_1\cup[y_1,y_2],$ then the whole curve $\gamma_2$ has to lie in $\Omega_1\cup[y_1,y_2].$) Let us assume that
$
  \Omega\subset\Omega_2,$ and in particular $|\Omega|\leq|\Omega_2|.$
Let us define $\Omega_2'=\Omega_2-x_0$ and $\Omega'=\Omega-x_0$.
It now follows from Step 1 applied to $\Omega_2',$ from Lemma 
\ref{lemma:p bigger 0 straight line minimizes} and from \eqref{eq:proof:shifting arg. x norm power p decreases} that
\begin{align*}
 \left(\frac{|\Omega|}{\pi}\right)^{\frac{p+1}{2}}\leq \left(\frac{|\Omega_2'|}{\pi}\right)^{\frac{p+1}{2}}\leq \int_{\partial \Omega_2'}|x|^pd\sigma 
 \leq \int_{\partial\Omega'}|x|^pd\sigma \leq\intdelomega |x|^p d\sigma.
\end{align*}
This proves the claim in the present case.
\smallskip 

\textit{Case 2: $0\in E.$} The argument is very similar to that of Case 1 and we omit the details: Since $0\in E,$ there exists $y_1,y_2\in\partial\Omega$ such that $0\in[y_1,y_2]$ and $(y_1,y_2)\cap\Omegabar=\emptyset$ (where $(y_1,y_2)$ is the open line segment). We define $\gamma_1,\gamma_2,$ respectively $\Omega_1,\Omega_2$ as in Case 1. Using that neither $\gamma_1$ nor $\gamma_2$ intersect $(y_1,y_2),$ that $\gamma_1$ and $\gamma_2$ do not intersect and that $0\notin\Omegabar,$ one easily deduces that $\Omega$ must be contained in $\Omega_1$ or $\Omega_2$. We then argue exactly as in Case 1, whereby we set $x_0=0,$ i.e. there is no need to shift the domain to origin.

\smallskip
\textit{Step 3.} We now show that we can also drop the assumption that $\Omega$ is connected. In view of Theorem \ref{theorem:Betta et alt} we can assume that $0\leq p\leq 1.$ Let $\Omega_i,$ $i=1,\ldots,l$ denote the connected components of $\Omega,$ and $R,R_i>0$ be such that $\pi R_i^2=|\Omega_i|,$ respectively $\pi R^2=|\Omega|.$ We know from Step 2 that for each $i=1,\ldots,l$
\begin{equation}
 \label{eq:eq true for disjoint components}
  2\pi R_i^{p+1}\leq\int_{\delomega_i}|x|^p d\sigma.
\end{equation}
It follows from $|\Omega|=\sum_{i=1}^l|\Omega_i|$ that
$
  R^{p+1}=\left(R_1^2+\cdots +R_l^2\right)^{\frac{p+1}{2}}.
$
The result now follows from \eqref{eq:eq true for disjoint components} and the inequality
\begin{equation}\label{eq:inequality thm 19 in HardyLiPo.}
  \left(R_1^2+\cdots +R_l^2\right)^{\frac{p+1}{2}}\leq \left(R_1^{p+1}+\cdots+R_l^{p+1}\right),
\end{equation}
see for instance \cite{HardyLittlewoodPolya} Theorem 19 page 28, where we have used that $0\leq p\leq 1.$ 
\end{proof}

\smallskip

We now conclude the proof of the main theorem.

\begin{proof}[Proof (Theorem \ref{theorem:main theorem weighted isop. inequlity} Part (iii) and (iv).] 
\textit{Step 1 (Proof of Part (iii)).} 
Step 4 in the proof of Part (i) shows that we cannot have equality if $\Omega$ is not simply connected. Thus we assume that $\Omega$ is simply connected.
If there is equality in the theorem, then  we must have equality in both of the inequalities in \eqref{eq:main thm proof step 4}. We obtain from the first one, equations \eqref{eq:proof:Ar goes to omega} and \eqref{eq:proof of main theorem. Final Ar} that
$$
  \lim_{r\to 1}\sum_{n=1}^{\infty} n |a_n|^2 r^{2n}=0.
$$
Therefore $a_n=0$ for all $n\geq 1,$ $Q(z)=\lambda+a_0 z$  and
$$
  g(z)=\frac{\lambda+a_0 z}{z}:\partial B_1\to \partial\Omega.
$$
It is easy to check that $g$ is a M\"obius transformation which sends the unit circle to a circle with center $a_0$ and radius $|\lambda|.$ This proves that $\Omega$ must be a ball. Note that $|\lambda|>|a_0|,$ since by assumption $0\in\Omega.$ It remains to show that $a_0=0,$ if $p\neq 0.$

We now use that also the second inequality in \eqref{eq:main thm proof step 4} must be an equality. Since $Q(z)=\lambda+a_0 z$ and $u(z)=-\lambda$ are holomorphic in $\mathbb{C}$ we have that \eqref{eq:proof:estimate of tau 0} holds now also for $r=1,$ and we must have equality. We therefore obtain from \eqref{eq:tau at zero and cauchy formula} and \eqref{eq:proof:estimate of tau 0} that the following equalities must hold
$$
  |\lambda|^{p+1}=|\tau(0)|=\frac{|\lambda|}{2\pi}\left|\int_0^{2\pi}\exp(p\varphi(e^{it}))dt\right| =\frac{|\lambda|}{2\pi}\int_0^{2\pi}|\exp(p\varphi(e^{it}))|dt.
$$
Defining $f(t)=\exp(p\varphi(e^{it})),$ the previous equality implies that 
$$
  \left|\int_0^{2\pi}f(t)dt\right|=\int_0^{2\pi}|f(t)|dt.
$$
This is only possible if there is a function $\phi:[0,2\pi]\to\re$ and two constants $c_1,c_2\in\re,$ such that $\Real(f(t))=c_1\phi(t)$ and $\Imag(f(t))=c_2\phi(t).$ (cf. for instance \cite{HardyLittlewoodPolya}, Theorem 201 page 148 applied to $f_1=\Real (f)$ and $f_2=\Imag (f)$). We thus get that $\varphi$ has to satisfy the two equations, defining $c=c_1+i c_2,$
$$
  \exp(p\varphi(e^{it}))=c\phi(t)\quad\text{ and }\quad\exp(\varphi(e^{it}))=\lambda+a_0 e^{it}.
$$
We multiply the second equation by $p$ and obtain by deriving the two equations
$$
  \frac{\phi'(t)}{\phi(t)}=p\varphi'(e^{it})i e^{it}=ip\frac{a_0 e^{it}}{\lambda+ a_0e^{it}}.
$$
Thus if $p\neq 0$ we must have that
$$
  \frac{a_0 e^{it}}{\lambda+a_0 e^{it}}=\in i\re\quad\text{ for all }t\in[0,2\pi].
$$
This is only possible if $a_0=0.$ Recall that the image of $\partial B_1$ under the M\"obius transformation $g$ is a circle and the image of that circle (containing the origin) under the map $z\to 1/z$ cannot be line.
\smallskip

\textit{Step 2 (Proof of Part (iv)).}
We can again assume that $p<1,$ by infering to \cite{Betta and alt.} Theorem 4.3, where it is proven for the case $p\geq 1$ that $\Omega$ has to be a ball centered at the origin.
Note first that if $0\leq p<1,$ then we have equality in \eqref{eq:inequality thm 19 in HardyLiPo.} if and only if $l=1$ (see e.g. \cite{HardyLittlewoodPolya} Theorem 19). Thus $\Omega$ has to be connected. Moreover, as in (ii), we must have that $\Omega$ must be simply connected. We now distinguish two cases:
If $0\in\Omega,$ then the result follows from Part (iii).
If $0\notin\Omegabar,$ then Step 2 in the proof Theorem \ref{theorem:main theorem weighted isop. inequlity} Part (ii), shows that there exists a domain $\Omega_2'\neq \Omega,$ which has piecewise $C^1$ boundary and  the properties
$$
  |\Omega|<|\Omega_2'|\quad\text{ and }\quad \int_{\delomega_2'}|x|^pd\sigma\leq \int_{\delomega} |x|^pd \sigma.
$$
We therefore cannot have equality in \eqref{theorem:eq:main isop. ineq.} for $\Omega.$
\end{proof}

\section{Applications: Hardy-Sobolev Inequality and Flucher's Estimate for $|\Omega|$}

In what follows $C_c^{\infty}(\re^2)$ shall denote the space of smooth functions with compact support. 
It is known, see for instance Caffarelli-Kohn-Nirenberg \cite{Caffarelli-Kohn-Nirenberg} that there exists a constant $C>0$ such that
\begin{equation}
 \label{section intro. Caff-Kohn-Ni ineq.}
  \|u\,|x|^{\gamma}\|_{L^r}\leq C\|Du\,|x|^{\alpha}\|_{L^1}\quad\text{ for all }u\in C_c^{\infty}(\re^2),
\end{equation}
if and only if $\alpha,\gamma$ and $r$ satisfy the conditions
\begin{equation}
 \label{eq:conditions in Caff-Kohn-Ni ineq.}
  r>0,\quad 0\leq\alpha-\gamma\leq 1,\quad 0<\frac{1}{r}+\frac{\gamma}{2}=\frac{\alpha+1}{2}.
\end{equation}
If $\Omega\subset\re^2$ is some smooth set we can take $u_i$ as a smooth approximation of the characteristic function $\chi_{\Omega}$ and obtain by a limiting process (or alternatively generalize \eqref{section intro. Caff-Kohn-Ni ineq.} to functions of bounded variation) that
$$
  \left(\int_{\Omega}|x|^{\gamma r}dx\right)^{\frac{1}{r}}\leq C\intdelomega |x|^{\alpha}d\sigma(x).
$$
We set $p=\alpha$ and $q=\gamma r,$ which satisfy, in view of \eqref{eq:conditions in Caff-Kohn-Ni ineq.},
$$
  0<2+q=r(p+1),\qquad 0\leq pr-q\leq r.
$$
From the first condition we immediately get that $q>-2$ and $p>-1.$ Moreover solving the first condition for $r$ and setting into the second one gives that $p-1\leq q\leq 2 p.$ Thus we have obtained that any $\Omega$ satisfies 
\begin{equation}
 \label{section intro Hardy-Sob. twice weighted isop. ineq.}
  \left(\int_{\Omega}|x|^q dx\right)^{\frac{p+1}{q+2}}\leq C\intdelomega |x|^pd\sigma(x)
\end{equation}
if
\begin{equation}
 \label{section intro Hardy-Sob. conditions on p and q.}
  -2<p-1\leq q\leq 2p.
\end{equation}
Note that the constant $C$ is the same as in \eqref{section intro. Caff-Kohn-Ni ineq.} and also that the conditions $p>-1$ and $q>-2$ imply that the integrals are finite for any domain $\Omega\subset\re^2.$ The next proposition shows that the converse is also true if $r\geq 1.$ 

\begin{proposition}
\label{proposition:equiv. Hardy sob and isop. twice weighted.}
Let $p,q\in\re$ be such that $p>-1,$ $q>-2$ and
$$
  \frac{q+2}{p+1}=r\geq 1.
$$
Suppose that there exists a constant $C$ such that \eqref{section intro Hardy-Sob. twice weighted isop. ineq.} holds for all bounded open smooth sets $\Omega\subset\re^2.$ Then the following inequality holds
$$
  \|u\,|x|^{\gamma}\|_{L^r}\leq C\|Du\,|x|^{\alpha}\|_{L^1}\quad\text{ for all }u\in C_c^{\infty}(\re^2),
$$
where $\gamma=q/r,$ $\alpha=p$ and the constant $C$ is the same as in \eqref{section intro Hardy-Sob. twice weighted isop. ineq.}.
\end{proposition}

In the proof we follow Struwe \cite{Struwe} page 43, generalizing thereby the case $p=q=0.$ Note that if $q=0,$ the Minkowsky inequality in the below proof is presicely valid for $p\in(-1,1].$  We use the notation: if $u\in C^{\infty}_c(\re^2),$ then $\Omega(t)=\Omega_u(t)=\{x\in\re^2;\,|u(x)|>t\}$ shall denote the level sets of $u.$ 
\smallskip

\begin{proof}
Let $\chi_{\Omega(t)}:\re^2\to \{0,1\}$ denote the characteristic function of $\Omega(t).$ 
Let $M=\max|u|.$ Then we can write for all $x\in\re^2$
$$
  |u(x)|=\int_0^M\chi_{\Omega(t)}(x) dt.
$$
Thus we obtain that
$$
  \|u\,|x|^{\gamma}\|_{L^r}=\left(\int_{\re^2}\left(\int_0^M\chi_{\Omega(t)}(x)|x|^{\gamma}dt\right)^rdx\right)^{\frac{1}{r}}.
$$
Using that $r\geq 1$ we can apply the Minkowsky inequality (see for instance Hardy Littlewood and P\'olya \cite{HardyLittlewoodPolya}, Theorem 202, page 148), and we get
\begin{align}
 \label{eq:proof imbedding:using minkowski ineq}
  \|u\,|x|^{\gamma}\|_{L^r}\leq\int_0^M\left(\int_{\re^2}\chi_{\Omega(t)}(x)|x|^{\gamma r}dx\right)^{\frac{1}{r}}dt
  =\int_0^M\left(\int_{\Omega(t)}|x|^q dx\right)^{\frac{p+1}{q+2}}dt.
\end{align}
We obtain from the hypothesis, namely inequality \eqref{section intro Hardy-Sob. twice weighted isop. ineq.}, that
\begin{equation}
 \label{eq:proof of imbedding:using isop ineq twice weighted.}
  \left(\int_{\Omega(t)}|x|^q dx\right)^{\frac{p+1}{q+2}}\leq
  C\int_{\partial\Omega(t)}|x|^{\alpha}d\sigma=C\int_{\{|u|=t\}}|x|^{\alpha}d\sigma.
\end{equation}
We substitute this inequality into the previous one
$$
  \|u\,|x|^{\gamma}\|_{L^r}\leq  C\int_0^M\left(\int_{\{|u|=t\}}|x|^{\alpha}d\sigma(x)\right)dt= C\int_0^{\infty}\left(\int_{\{|u|=t\}}|x|^{\alpha}d\sigma(x)\right)dt.
$$
We now apply the coarea formula to get
$$
  \|u\,|x|^{\gamma}\|_{L^r}\leq C \int_{\re^2}|\nabla|u(x)||\,|x|^{\alpha}dx=C \int_{\re^2}|\nabla u(x)|\,|x|^{\alpha}dx,
$$
which proves the proposition.
\end{proof}
\smallskip

%%%%%%%%%%%%%%%%%%%%%%%%%%%%%%%%%%%%%%%%%%%%%%%%%%%%%%%%%
We now apply Theorem \ref{theorem:main theorem weighted isop. inequlity} and the previous proposition to obtain the following imbedding with best constant. This is just one example: one can now easily obtain the best constant in other Hardy-Sobolev inequalities in exactly the same way, for instance using \cite{Diaz-Harman-Howe-Thompson} Proposition 4.21, instead of Theorem \ref{theorem:main theorem weighted isop. inequlity}.

\begin{theorem}
\label{theorem:sobolev imbedding with singular weight}
Let $p\in (-1,1]$ and define $r$ as
$$
  r=\frac{2}{p+1}.
$$
Consider the inequality
\begin{equation}
 \label{theorem:equation:singular weighted imbedding}
  \|u\|_{L^r(\re^2)}\leq \frac{\pi^{1/r}}{2\pi}\int_{\re^2}|\nabla u(x)|\,|x|^p dx.
\end{equation}
Then we have the following statements:
\smallskip

(i) If $p\in [0,1]$ the inequality \eqref{theorem:equation:singular weighted imbedding} holds for all $u\in C_c^{\infty}(\re^2).$ 
\smallskip

(ii) If $p\in (-1,0)$ inequality \eqref{theorem:equation:singular weighted imbedding} holds for all $u\in C_c^{\infty}(\re^2)$ 
with the property that for all $t\in (0,\max|u|)$ the sets $\Omega(t)$ are connected and $0\in\Omega(t).$
\smallskip

(iii) In both cases the constant $\pi^{\frac{1}{r}}/2\pi$ is optimal and it is not attained.
\end{theorem}

\begin{remark}
In the embedding (ii) of Theorem \ref{theorem:sobolev imbedding with singular weight} neither of the hypothesis $0\in\Omega(t)$ nor
$\Omega(t)$ is connected can be relaxed.
\end{remark}

\begin{proof}
Part (i) and (ii) follow directly from Theorem \ref{theorem:main theorem weighted isop. inequlity} and Proposition \ref{proposition:equiv. Hardy sob and isop. twice weighted.}. It remains to see that the constant in the  imbedding is optimal. Note that we have equality in \eqref{eq:proof imbedding:using minkowski ineq} if and only if $\chi_{\Omega}(t)$ can be factorized as $\chi_{\Omega}(t)=\Phi(x)\Psi(t)$ (see \cite{HardyLittlewoodPolya}, Theorem 202). And in \eqref{eq:proof of imbedding:using isop ineq twice weighted.} we have equality if $\Omega(t)$ is a ball centered at the origin (unless $0\in\partial\Omega(t)).$ This suggests that one should take $u$ as the characteristic function of a ball. We proceed thus by approximating $\chi_{B_R(0)}$ for some $R>0.$ These arguments are standard and we omit the details
\end{proof}
\smallskip

Another application of Theorem \ref{theorem:main theorem weighted isop. inequlity} is a generalization of an estimate by Flucher \cite{Flucher} for $|\Omega|$ in terms of the Green's function (see Theorem 17 in \cite{Flucher}). The Green's function for $\Omega$ with singularity at $x\in\Omega$ will be denoted by $G_{\Omega,x}.$ 

\begin{theorem}
\label{theorem:upper bound for R by Greens function with weight}
Let $\Omega\subset\re^2$ be a bounded open set with smooth boundary $\delomega,$ and $x\in\Omega$. Then the inequality 
$$
  |\Omega|^{1-\frac{\beta}{2}}\leq \frac{1}{4\pi^{1+\frac{\beta}{2}}} \int_{\delomega} \frac{1}{|y|^\beta|\nabla G_{\Omega,x}(y)|}d\sigma(y).
$$
holds true

(i) for all $0\leq \beta\leq 2$ if $\Omega$ is connected and $0\in\Omega.$

(ii) for all $\beta\leq 0$ without restriction on $\Omega.$
\end{theorem}

In the proof we will use the following simple lemma. 

\begin{lemma}
\label{lemma:properties of Green's function}
$G_{\Omega,x}$ satisfies for every $t\in[0,\infty)$ the two equations:
$$
  \int_{\{G_{\Omega,x}<t\}}\left|\nabla G_{\Omega,x}(y)\right|^2 dy=t
  \;\text{ and }\;
  \int_{\{G_{\Omega,x}=t\}}\left|\nabla G_{\Omega,x}(y)\right| d\sigma(y)=1.
$$
\end{lemma}

\begin{proof} By the definition of Green's function $\int_{\Omega}\nabla G_{\Omega,x}(y)\nabla f(y)dy=f(x)$ for all $f\in W^{1,2}_0(\Omega).$ In particular chosing $f(y)=\inf\{G_{\Omega,x}(y),t\}$ gives the first equality. From the coarea formula we have
$$
  \int_{\{G_{\Omega,x}<t\}}|\nabla G_{\Omega,x}(y)|^2 dy=\int_0^t\left(\int_{\{G_{\Omega,x}=s\}}|\nabla G_{\Omega,x}(y)|d\sigma(y)\right)ds.
$$
The second equality therefore follows from the first one by derivation.
\end{proof}
\smallskip

\begin{proof}[Proof (Theorem \ref{theorem:upper bound for R by Greens function with weight}).] Set $p=-\beta/2$ and let $R>0$ be such that $|\Omega|=\pi R^2.$ Then we get from Theorem \ref{theorem:main theorem weighted isop. inequlity}, 
$$
  2\pi R^{1+p}\leq \intdelomega |y|^pd\sigma(y)=\intdelomega \frac{\sqrt{|\nabla G_{\Omega,x}(y)|}}{\sqrt{|y|^{\beta}|\nabla G_{\Omega,x}(y)|}} d\sigma(y).
$$
We apply H\"older inequality and the second equality in Lemma \ref{lemma:properties of Green's function} with $t=0$  to get that
\begin{align*}
  2\pi R^{1+p}\leq &\left(\intdelomega |\nabla G_{\Omega,x}(y)|d\sigma(y)\right)^{\frac{1}{2}} \left(\intdelomega \frac{1}{|y|^{\beta}|\nabla G_{\Omega,x}(y)|}d\sigma(y)\right)^{\frac{1}{2}}
  \smallskip \\
  =& \left(\intdelomega \frac{1}{|y|^{\beta}|\nabla G_{\Omega,x}(y)|}d\sigma(y)\right)^{\frac{1}{2}}.
\end{align*}
From this the theorem follows immediately.
\end{proof}

\bigskip

\noindent\textbf{Acknowledgements} I have benefitted from helpful discussions and comments from A. Adimurthi, F. Morgan, P. Roy and K. Sandeep.


\begin{thebibliography}{200}
\small{
%\bibitem{Adimurthi}Adimurthi A., Filippas S. and Tertikas A., On the best constant %of Hardy-Sobolev inequalities, \textit{Nonlinear Anal.} \textbf{70} (2009), no. 8. %2826--2833.

\bibitem{Adi-Sandeep} Adimurthi A. and Sandeep K., A singular Moser-Trudinger embedding and its applications, \textit{NoDEA Nonlinear Differential Equations Appl.}, \textbf{13} (2007), no. 5-6, 585--603.
\vspace{-.25cm}

\bibitem{Betta and alt.}Betta M.F., Brock F., Mercaldo A. and Posteraro M.R., A weighted isoperimetric inequality and applications to symmetrization, \textit{J. of Inequal. and Appl.}, \textbf{4} (1999), 215--240.
\vspace{-.25cm}

\bibitem{Rosales Canete Bayle Morgan}Bayle V., Ca\~nete A. , Morgan F. and Rosales C., 
On the isoperimetric problem in Euclidean space with density,
\textit{Calc. Var. Partial Differential Equations} \textbf{31} (2008), no. 1, 27--46.
\vspace{-.25cm}


\bibitem{Caffarelli-Kohn-Nirenberg}Caffarelli L., Kohn R. and Nirenberg L., First order interpolation inequalities with weights, \textit{Compositio Math.}, 53 (1984), no. 3, 259--275.
\vspace{-.25cm}

\bibitem{Canete-Miranda-Wittone}Ca\~nete A., Miranda M. and Vittone D., 
Some isoperimetric problems in planes with density, \textit{
J. Geom. Anal.} \textbf{20} (2010), no. 2, 243--290. 
\vspace{-.25cm}



\bibitem{Carleson-Chang}Carleson L. and Chang S.-Y. A., On the existence of an extremal function for an inequality by J. Moser, \textit{Bull. Sci. Math.}, (2) 110 (1986), no. 2, 113--127.
\vspace{-.25cm}

\bibitem{Carroll-Jacob-Quinn-Walters}Carroll C., Jacob A. Adam, Quinn C. and Walters R.,
The isoperimetric problem on planes with density, \textit{
Bull. Aust. Math. Soc.} \textbf{78} (2008), no. 2, 177--197.
\vspace{-.25cm}

\bibitem{Catrina-Wang}Catrina F. and Wang Z.Q., On the Caffarelli-Kohn-Nirenberg inequalities: sharp constants, existence (and nonexistence), and symmetry of extremal functions, \textit{ Comm. Pure Appl. Math.} \textbf{54} (2001), no. 2, 229--258.
\vspace{-.25cm}

\bibitem{Chou and Chu}Chou K.S. and Chu C.W.,
On the best constant for a weighted Sobolev-Hardy inequality, \textit{
J. London Math. Soc.} (2) \textbf{48} (1993), no. 1, 137--151.
\vspace{-.25cm}

\bibitem{CorderoVillani et alt.}Cordero-Erausquin D., Nazaret B. and Villani C.,
A mass-transportation approach to sharp Sobolev and Gagliardo-Nirenberg
inequalities.
\textit{Adv. Math.} \textbf{182} (2004), no. 2, 307--332.
\vspace{-.25cm}

\bibitem{Csato-Roy}Csat\'o G. and Roy P., Extremal functions for the singular Trudinger-Moser inequality in 2 dimensions, preprint,  arXiv:1410.8638.
\vspace{-.25cm}

\bibitem{Dahlberg-Dubbs-Newkirk-Tran}Dahlberg J., Dubbs A., Newkirk E. and Tran H., 
Isoperimetric regions in the plane with density $r^p$, \textit{
New York J. Math.} \textbf{16} (2010), 31--51.
\vspace{-.25cm}

\bibitem{Diaz-Harman-Howe-Thompson} D\'iaz A., Harman N., Howe S. and  Thompson D., Isoperimetric problems in sectors with
density, \textit{Adv. Geom.} \textbf{12} (2012), 589--619.
\vspace{-.25cm}

\bibitem{Federer-Fleming} 
Federer H., Fleming W.H., Normal and integral currents, \textit{ Ann. of Math.} (2) \textbf{72} (1960), 458--520.
\vspace{-.25cm}

\bibitem{Figalli-Maggi.LogConvex}Figalli A. and Maggi F., 
On the isoperimetric problem for radial log-convex densities, 
\textit{Calc. Var. Partial Differential Equations}, \textbf{48} (2013), no. 3-4, 447--489.
\vspace{-.25cm}

\bibitem{Figalli-Maggi-Pratelli}Figalli A., Maggi, F. and Pratelli, A.,
A mass transportation approach to quantitative isoperimetric inequalities, \textit{
Invent. Math.} \textbf{182} (2010), no. 1, 167--211.
\vspace{-.25cm}

\bibitem{Fleming-Rishel} Fleming W.H. and Rishel R., An integral formula for total gradient variation, \textit{ Arch. Math.} (Basel), \textbf{11} (1960), 218--222.
\vspace{-.25cm}

\bibitem{Flucher}Flucher M., Extremal functions for the Trudinger-Moser inequality in 2 dimensions, \textit{Comment. Math. Helvetici}, \textbf{67} (1992), 471--497.
\vspace{-.25cm}

\bibitem{Fusco-Maggi-Pratelli}Fusco N., Maggi F. and Pratelli A.,
On the isoperimetric problem with respect to a mixed Euclidean-Gaussian
density, 
\textit{J. Funct. Anal.} 260 (2011), no. 12, 3678--3717.
\vspace{-.25cm}


\bibitem{HardyLittlewoodPolya} Hardy G.H., Littlewood J.E. and P\'olya G, \textit{Inequalities}, Second Edition, Cambridge University Press, 1952.
\vspace{-.25cm}

\bibitem{Horiuchi}Horiuchi T., 
Best constant in weighted Sobolev inequality with weights being powers of
distance from the origin,
\textit{J. Inequal. Appl.} \textbf{1} (1997), no. 3, 275--292.
\vspace{-.25cm}

\bibitem{Krantz Complex Analysis}Krantz S.G.,\textit{Geometric Function Theory, Explorations in Complex Analysis}, Cornerstones, Birkh\"auser Boston, Inc., Boston, MA, 2006.
\vspace{-.25cm}

\bibitem{Lieb}Lieb E. H.,
Sharp constants in the Hardy -- Littlewood -- Sobolev and related inequalities.
\textit{Ann. of Math.} (2) \textbf{118} (1983), no. 2, 349--374.
\vspace{-.25cm}


\bibitem{Malchiodi-Martinazzi}Malchiodi A. and Martinazzi L.,
Critical points of the Moser-Trudinger functional on a disk, \textit{
J. Eur. Math. Soc. (JEMS)}, \textbf{16} (2014), no. 5, 893--908.
\vspace{-.25cm}


\bibitem{Mateljevic-Pavlovic}Mateljevic M. and Pavlovic M., New proofs of the isoperimetric inequality and some generalizations, \textit{J. Math. Anal. Appl.}, \textbf{98} (1984), no. 1, 25--30. 
\vspace{-.25cm}

\bibitem{Morgan-Regularity}Morgan F., Regularity of isoperimetric hypersurfaces in Riemannian manifolds, \textit{Trans. Amer. Math. Soc.}, \textbf{355} (2003), no. 12, 5041--5052.
\vspace{-.25cm}

\bibitem{Morgan-Pratelli}Morgan F. and Pratelli A., Existence of isoperimetric regions in $\mathbb{R}^n$ with density, \textit{ Ann. Global Anal. Geom.} \textbf{43} (2013), no. 4, 331--365. 
\vspace{-.25cm}



\bibitem{Struwe}Struwe M.,\textit{Variational methods, applications to nonlinear partial differential equations and Hamiltonian systems}, Fourth edition, Springer-Verlag, Berlin, 2008.
\vspace{-.25cm}


\bibitem{Talenti} Talenti G., Best constant in Sobolev inequality, \textit{Ann. Mat. Pura Appl.} (4) \textbf{110} (1976), 353--372.
\vspace{-.25cm}



}
\end{thebibliography}
\end{document}